\documentclass[11pt,a4paper]{amsart}

\usepackage[T1]{fontenc}
\usepackage[utf8]{inputenc}
\usepackage{lmodern}
\usepackage{microtype}
\usepackage{amsmath,amssymb,amsthm}
\usepackage[margin=1.15in]{geometry}
\usepackage{xcolor}
\usepackage{hyperref}

\input{glyphtounicode}
\microtypesetup{protrusion=true,expansion=false}
\theoremstyle{plain}
\newtheorem{theorem}{Theorem}[section]
\newtheorem{proposition}[theorem]{Proposition}
\newtheorem{lemma}[theorem]{Lemma}
\newtheorem{corollary}[theorem]{Corollary}

\theoremstyle{definition}
\newtheorem{computation}[theorem]{Computation}
\newtheorem{observation}[theorem]{Observation}
\newtheorem{conjecture}[theorem]{Conjecture}
\newtheorem{remark}[theorem]{Remark}

\newcommand{\PR}{\mathbf{P}}
\newcommand{\EX}{\mathbf{E}}
\newcommand{\Var}{\operatorname{Var}}
\newcommand{\dg}{\operatorname{deg}}
\newcommand{\Nc}{N}
\newcommand{\ADJ}{\mathrm{(ADJ}_n\mathrm{)}}
\newcommand{\DIS}{\mathrm{(DIS}_n\mathrm{)}}

\title[Negative correlation in uniform forests of $K_n$]
  {Edges of the uniform random forest of $K_n$ are pairwise negatively
  correlated for every $n$}
\author{Anish Gupta}
\address{Independent researcher}
\email{ag2269@cantab.ac.uk}
\urladdr{https://orcid.org/0009-0008-8137-7729}
\date{7 August 2026}
\subjclass[2020]{Primary 60C05; Secondary 05C30, 05C80}
\keywords{uniform random forest, negative correlation, complete graph,
  labelled forest, component count}

\hypersetup{
  colorlinks=true,
  linkcolor=blue!45!black,
  citecolor=blue!45!black,
  urlcolor=blue!45!black,
  pdftitle={Edges of the uniform random forest of K\_n are pairwise negatively correlated for every n},
  pdfauthor={Anish Gupta},
  pdfsubject={Pairwise negative correlation in uniform forests of complete graphs},
  pdfkeywords={uniform random forest, negative correlation, complete graph, labelled forest},
  pdfdisplaydoctitle=true
}

\begin{document}
\begin{abstract}
Let $F_n$ be uniform on all forests of the simple complete graph $K_n$, with
isolated vertices allowed. A conjecture of Kahn and of Winkler, studied by
Grimmett and Winkler, asserts that any two distinct edges of any finite graph
are negatively correlated under the uniform forest measure. Stark proved this
for $G=K_n$ once $n$ is sufficiently large, but did not furnish an explicit
threshold. We prove it for every $n\ge2$, strictly whenever two
distinct edges exist. The difficulty is concentrated in the disjoint-edge
orbit, whose correlation ratio tends to one. We remove this cancellation
before estimating anything: the desired inequality becomes an exact comparison
among the first two moments of the component count and the expected sum of
squared degrees. When the component count fluctuates, this comparison contains
a variance term absent from the fixed-component identities of Tang and Zhang.
Component marking, tail elimination, and effective Stirling bounds control the
three moments for $n\ge651$; exact integer recurrences cover the remaining
values.
\end{abstract}

\maketitle

\section{Introduction}

Let $G$ be a finite graph and let $F$ be a uniformly random \emph{forest} of $G$:
a uniformly random acyclic subset of the edge set, with isolated vertices
permitted, the empty set included, and no connectivity or spanning requirement.
Does the presence of one edge in $F$ make a second edge less likely to be
present?

For the uniform \emph{spanning tree} the answer is yes, and classically so: the
transfer-current theorem makes the edge indicators determinantal, and they are
negatively associated. This depends on graphic structure: bases of arbitrary
matroids need not even be pairwise negatively correlated \cite{HSW2022}. The
uniform forest measure is not governed by the spanning-tree determinant, and
there the question is open.

\begin{conjecture}[Kahn's Conjecture~10.11 \cite{Kahn2000}, Winkler; see Pemantle \cite{Pemantle2000}, Grimmett--Winkler \cite{GW2004}, Stark \cite{Stark2011}]
\label{conj:gw}
For every finite graph $G$ and every two distinct edges $\varepsilon_1,\varepsilon_2$
of $G$,
\[
  \PR(\varepsilon_1 \in F,\ \varepsilon_2 \in F)
  \;\le\; \PR(\varepsilon_1 \in F)\,\PR(\varepsilon_2 \in F).
\]
\end{conjecture}

We call the displayed inequality \emph{pairwise negative correlation} (p-NC).
Pemantle \cite{Pemantle2000} records the question as one posed by Winkler.
Grimmett and Winkler \cite{GW2004} verified it for all simple graphs on at most
$8$ vertices and all simple graphs on $9$ vertices with at most $18$ edges; Semple and
Welsh \cite{SW2008} later proved weighted forms for an infinite class of graphs
and matroids. For the complete graph, Stark \cite{Stark2011} proved p-NC for all
sufficiently large $n$ by singularity analysis \cite{FS2009} of the tree and
forest generating functions. He records that no explicit value of the eventual
threshold is furnished (\cite{Stark2011}, p.~531); the singularity-analysis
remainders carry no numerical constants.

The uniform measure is also the $\beta=1$ arboreal gas. More generally,
activity-$\beta$ arboreal gas is Bernoulli bond percolation with
$p=\beta/(1+\beta)$ conditioned on acyclicity, and is obtained from the
random-cluster model as $q\to0$ with $p=\beta q$ \cite{BCHS2021}. Huang
\cite{Huang2024} proves
negative-correlation results on complete graphs in small- and large-$\beta$
regimes, but not at $\beta=1$ for every $n$; Nguyen and Pylyavskyy
\cite{NP2025} give a related combinatorial formula for random-cluster
correlations at $q=1$.

Tang and Zhang \cite{TZ2026a} prove large-$n$ results for three neighboring
uniform measures on $K_n$: connected spanning subgraphs by disconnection
estimates, forests with a fixed number of components from an explicit counting
formula and moment inequalities, and connected subgraphs of fixed excess by
singularity analysis. Their Remark~1.7 also records that p-NC does not pass
between a measure and its truncations, so the fixed-component theorem does not
settle the unconditioned forest measure. Their fixed-component moment identities
and a later spanning-tree formulation \cite{TZ2026b} are directly relevant to
the reduction used below; the precise comparison is made after
Proposition~\ref{prop:moments}.

Stark's theorem and the finite census therefore did not meet: without an
explicit threshold, no finite computation could close the gap between them.
This note supplies a threshold and verifies every case below it.

Throughout, $K_n$ is the \textbf{simple} complete graph on the vertex set $[n]$,
and $F_n$ is uniform on all forests of $K_n$.

\begin{theorem}\label{thm:main}
For every integer $n \ge 2$ and every pair of distinct edges $e \ne f$ of $K_n$,
\[
  \PR(e \in F_n \text{ and } f \in F_n)
  \;\le\; \PR(e \in F_n)\,\PR(f \in F_n).
\]
\end{theorem}

\subsection{Two integer inequalities}

Because $\operatorname{Aut}(K_n) = S_n$ is transitive on edges and on each of the
two orbits of pairs of distinct edges, Theorem~\ref{thm:main} is equivalent to a
pair of inequalities between integers. Write
\[
  A(n) = \#\{\text{forests of } K_n\},\qquad
  \Nc_e(n),\ \Nc_{\mathrm{adj}}(n),\ \Nc_{\mathrm{dis}}(n)
\]
for the number of forests containing, respectively, one fixed edge, a fixed pair
of adjacent edges, and a fixed pair of disjoint edges. Then
Theorem~\ref{thm:main} says exactly
\begin{align}
  \Nc_{\mathrm{adj}}(n)\,A(n) &\le \Nc_e(n)^2 && (n \ge 3), \tag{$\mathrm{ADJ}_n$}\\
  \Nc_{\mathrm{dis}}(n)\,A(n) &\le \Nc_e(n)^2 && (n \ge 4). \tag{$\mathrm{DIS}_n$}
\end{align}
The cases $n=2$ (only one edge) and, for the disjoint orbit, $n=3$ (no disjoint
pair) are vacuous.

In fact both inequalities are strict in every nonvacuous finite case; the
disjoint ratio merely tends to $1$.

The first instance with both orbits already shows the difference between them.
For $n=4$ one has
$(A,\Nc_e,\Nc_{\mathrm{adj}},\Nc_{\mathrm{dis}})=(38,14,4,5)$, so the
joint-to-product ratios are $38/49$ for an adjacent pair and $95/98$ for a
disjoint pair.

\subsection{Why the disjoint inequality is delicate}\label{ss:delicate}

The two orbits are not of equal difficulty, and the reason is visible already in
the spanning-tree stratum of the forest measure.

\begin{observation}[Stark, Theorems~1.2 and 1.3 \cite{Stark2011}]\label{obs:ust}
Let $T$ be a uniform spanning tree of $K_n$, $n \ge 3$. Then
$\PR(\varepsilon \in T) = 2/n$; adjacent pairs satisfy
$\PR(\varepsilon_1,\varepsilon_2 \in T) = 3n^{-2}$; and nonadjacent pairs satisfy
$\PR(\varepsilon_1,\varepsilon_2 \in T) = 4n^{-2}
= \PR(\varepsilon_1\in T)\PR(\varepsilon_2 \in T)$, i.e.\ they are exactly
independent. Theorem 1.3 of \cite{Stark2011} is stronger still: for $n \ge 2s$,
any $s$ mutually nonadjacent edges of $K_n$ are jointly independent in $T$.
\end{observation}

Both statements also follow at once from Lemma~\ref{lem:gencayley} below,
which gives $2^s n^{n-s-2}$ spanning trees containing $s$ prescribed disjoint
edges.

For an adjacent pair the spanning tree already has a definite $3/4$ of slack, and
$\ADJ$ inherits it. For a disjoint pair the spanning tree sits exactly at
equality, so no slack whatever is available from that stratum, and what survives
in $\DIS$ is correspondingly thin. Stark quantified it.

\begin{observation}[Stark \cite{Stark2011}, Theorem 1.4]\label{obs:rate}
For the uniform forest of $K_n$,
$\PR(\varepsilon_1 \in F)\PR(\varepsilon_2 \in F)
 = 4n^{-2} - 4n^{-3} - 23n^{-4} + O(n^{-5})$ and, for nonadjacent pairs,
$\PR(\varepsilon_1,\varepsilon_2 \in F)
 = 4n^{-2} - 4n^{-3} - 27n^{-4} + O(n^{-5})$. Dividing, the disjoint ratio is
$1 - n^{-2} + O(n^{-3})$. For an adjacent pair the same theorem gives
$\PR(\varepsilon_1,\varepsilon_2\in F)=3n^{-2}+O(n^{-3})$, so its ratio tends
to $3/4$.
\end{observation}

So $\DIS$ is asymptotically tight, with relative slack of order $n^{-2}$ only.
Any argument that loses $\Theta(n^{-2})$ at any point cannot close the disjoint
orbit. The initial reduction must therefore be exact, and the subsequent
estimates must preserve the leading terms.

\subsection{Outline of the proof}\label{ss:outline}

The adjacent orbit is settled quickly (Section~\ref{sec:assembly}): comparing the
component-peeling recurrences for $\Nc_{\mathrm{adj}}$ and $\Nc_e$ term by term
gives $2(n-2)\Nc_{\mathrm{adj}} \le 3\Nc_e$, and $\ADJ$ then follows from any
bound $\EX[c_n] \le 2$ on the mean number of components. The disjoint orbit is
where the work lies, and the argument has three stages.

Let $c$ be the number of components of $F_n$, let $m=n-c$ be its number of
edges, and let $D=\sum_v\dg_{F_n}(v)^2$. Double counting expresses the three
edge counts through moments of $m$ and $D$. Proposition~\ref{prop:delta} then
makes $\DIS$ equivalent to $\Delta(n)\ge0$, an inequality involving only
$\EX[D]$, $\EX[c]$, and $\EX[c^2]$. Tang and Zhang obtained the corresponding
fixed-$m$ identities \cite{TZ2026a,TZ2026b}; here $m$ fluctuates, and the new
term $\Var(m)$ must be controlled. The equivalence removes the order-$n^{-2}$
cancellation before any estimate is made.

Component-marking identities reduce the needed bounds on $\EX[c]$ and
$\EX[c^2]$ to effective control of $a_n:=A(n)/n^{n-2}$. Robbins' form of
Stirling's inequality supplies that control. A tail-elimination step then absorbs
all but forty explicit terms of the second-moment identity into the already
bounded mean.

Finally, expanding $\sum_FD(F)$ over the component of a fixed vertex gives
$\EX[D_n]\ge5n-\tfrac{29}{2}$. The only inputs to the estimates from the
literature are Pr\"ufer's bijection and Robbins' inequality; Appendix
\ref{app:cert} records the exact rational verification of every constant.

\subsection{Scope, and the role of computation}\label{ss:scope}

Theorem~\ref{thm:main} concerns the single family $G=K_n$ of simple graphs and
the pairwise form of the conjecture. Conjecture~\ref{conj:gw} for general finite
graphs remains open, as do its multigraph form --- the one tied to the pairwise
$q<1$ random-cluster question --- full negative association of the uniform forest
measure, even for $K_n$, and the connected-subgraph analogue proposed in
\cite{GW2004}. The measure here is uniform on \emph{all} forests: it is not the
uniform spanning forest of the probabilistic literature, not the uniform spanning
tree, and not a forest conditioned to have a prescribed number of components.
Nothing below improves the constant $c=2$ in the universal bound
$\PR(\varepsilon_1,\varepsilon_2\in F)\le c\,\PR(\varepsilon_1)\PR(\varepsilon_2)$
valid for all graphs \cite{BH2020}.

Finite ranges are checked in exact integers. Each analytic estimate ends in a
rational expression with one-sided enclosures for $e$, $\pi$, $\sqrt{2\pi}$,
and $\zeta(3/2)$, so no floating-point value decides an inequality. The program
\texttt{tests/check.py} and its limits are described in Appendix~\ref{app:cert};
universal inequalities such as $\log(1-x)\le-x-x^2/2$ are proved in the text,
not delegated to computation.

\section{Notation, exact counts and identities}\label{sec:prelim}

\subsection{Notation}

\begin{center}
\begin{tabular}{ll}
$A(n)$ & number of forests of $K_n$; $A(0)=A(1)=1$ (OEIS A001858, \cite{OEIS})\\
$a_n$ & $A(n)/n^{n-2}$, with $1^{-1}=1$ and $2^{0}=1$, so $a_1=1$, $a_2=2$\\
$t(k)$ & $k^{k-2}$, the number of trees on a $k$-set; $t(1)=t(2)=1$\\
$c = c(F)$ & number of components of the forest $F$, isolated vertices counted\\
$m = |F|$ & number of edges of $F$; always $m = n - c$\\
$D = D(F)$ & $\sum_{v\in[n]} \dg_F(v)^2$\\
$\EX$ & expectation with respect to the uniform measure on the $A(n)$ forests\\
$\mu$ & $\EX[c_n]$\\
$\Nc_e,\Nc_{\mathrm{adj}},\Nc_{\mathrm{dis}}$ & as in the introduction\\
\end{tabular}
\end{center}

\begin{lemma}[Pr\"ufer \cite{Pruefer1918}]\label{lem:prufer}
For $k \ge 2$ there is a bijection between trees on $[k]$ and sequences in
$[k]^{k-2}$ under which $\dg_T(v) = 1 + (\text{number of occurrences of } v)$ for
every $v$. Consequently the number of trees on $[k]$ is $t(k)=k^{k-2}$, and the
number with prescribed degrees $d_1,\dots,d_k$ (all $\ge 1$, summing to $2k-2$)
is $\binom{k-2}{d_1-1,\dots,d_k-1}$.
\end{lemma}

\subsection{Trees through a prescribed forest}

Every count in this note is assembled by peeling off one component of a forest,
so we first need to know how many trees contain a given forest.

\begin{lemma}\label{lem:gencayley}
Let $S_1,\dots,S_r$ be a partition of $[k]$ with $|S_i|=s_i$, and let $\Phi$ be a
forest on $[k]$ whose components have vertex sets exactly $S_1,\dots,S_r$. Then
the number of trees $T$ on $[k]$ with $\Phi \subseteq T$ is
\[
  k^{\,r-2}\prod_{i=1}^{r} s_i ,
\]
read as $1$ when $r=1$.
\end{lemma}

\begin{proof}
A tree $T \supseteq \Phi$ is obtained by adding $r-1$ edges joining distinct
$S_i$'s so that the contracted multigraph on $\{1,\dots,r\}$ is a tree. For $r=1$
there is nothing to add. For $r \ge 2$, contract each $S_i$ to a node $i$: the
added edges form a tree $\tau$ on $[r]$, and if node $i$ has degree $d_i$ in
$\tau$ then the $d_i$ endpoints at $i$ may be chosen independently among the
$s_i$ vertices of $S_i$, in $s_i^{d_i}$ ways. Distinct choices give distinct
trees $T$, and every $T \supseteq \Phi$ arises once. By
Lemma~\ref{lem:prufer} the number of $\tau$ with degree sequence $(d_i)$ is
$(r-2)!/\prod_i (d_i-1)!$, so, writing $e_i = d_i-1 \ge 0$ with
$\sum_i e_i = r-2$,
\[
\begin{aligned}
  \#\{T \supseteq \Phi\}
  &= \sum_{(d_i)} \frac{(r-2)!}{\prod_i (d_i-1)!}\prod_i s_i^{d_i}\\
  &= \Big(\prod_i s_i\Big)\sum_{\substack{e_i \ge 0\\ \sum e_i = r-2}}
      \frac{(r-2)!}{\prod_i e_i!}\prod_i s_i^{e_i}\\
  &= \Big(\prod_i s_i\Big)\Big(\sum_i s_i\Big)^{r-2},
\end{aligned}
\]
by the multinomial theorem, and $\sum_i s_i = k$.
\end{proof}

Applying Lemma~\ref{lem:gencayley} to $\Phi$ consisting of one edge, of two
adjacent edges, and of two disjoint edges (plus isolated vertices) gives the
three counts we need:
\begin{equation}\label{eq:t234}
  t_2(k) = 2k^{k-3}\ (k\ge2),\qquad
  t_3(k) = 3k^{k-4}\ (k\ge3),\qquad
  t_4(k) = 4k^{k-4}\ (k\ge4),
\end{equation}
being the number of trees on a $k$-set containing a fixed edge, a fixed pair of
adjacent edges, and a fixed pair of disjoint edges respectively. The degenerate
values are $t_2(2)=1$, $t_3(3)=1$, $t_4(4)=4$, consistent with the rational
reading of the powers.

\subsection{Peeling recurrences}

\begin{proposition}\label{prop:recur}
For $n \ge 1$,
\begin{align}
  A(n) &= \sum_{k=1}^{n}\binom{n-1}{k-1}t(k)\,A(n-k), \qquad A(0)=1,
    \label{eq:recA}\\
  \Nc_e(n) &= \sum_{k=2}^{n}\binom{n-2}{k-2}t_2(k)\,A(n-k),
    \label{eq:recNe}\\
  \Nc_{\mathrm{adj}}(n) &= \sum_{k=3}^{n}\binom{n-3}{k-3}t_3(k)\,A(n-k),
    \label{eq:recNadj}\\
  \Nc_{\mathrm{dis}}(n) &= \sum_{k=4}^{n}\binom{n-4}{k-4}t_4(k)\,A(n-k)
       + \sum_{k=2}^{n-2}\binom{n-4}{k-2}t_2(k)\,\Nc_e(n-k).
    \label{eq:recNdis}
\end{align}
\end{proposition}

\begin{proof}
For \eqref{eq:recA}, classify forests by the vertex set $S \ni 1$ of the
component containing vertex $1$: there are $\binom{n-1}{k-1}$ choices of $S$ with
$|S|=k$, $t(k)$ trees on $S$, and $A(n-k)$ forests on $[n]\setminus S$.
For \eqref{eq:recNe} and \eqref{eq:recNadj}, do the same with $S$ the vertex set
of the component containing the required edge, resp.\ the required adjacent pair,
using \eqref{eq:t234}. For \eqref{eq:recNdis}, split according to whether the two
required disjoint edges lie in one component or in two. In the first case $S$ is
that component and \eqref{eq:t234} applies. In the second case, let $S$ be the
component containing the \emph{first} of the two edges: it contains that edge and
avoids the two endpoints of the second, giving $\binom{n-4}{k-2}$ choices and
$t_2(k)$ trees, and the remaining $n-k$ vertices carry a forest containing the
second edge, of which there are $\Nc_e(n-k)$. The two components are
distinguished by which required edge they contain, so no factor $1/2$ occurs.
\end{proof}

\subsection{Moment identities}

The next proposition is what converts the three pair counts into moments; it is
the entry point for the reduction of Section~\ref{sec:delta}.

\begin{proposition}\label{prop:moments}
For every $n \ge 4$,
\begin{align}
  n(n-1)\,\frac{\Nc_e(n)}{A(n)} &= 2\,\EX[m], \label{eq:mom1}\\
  n(n-1)(n-2)\,\frac{\Nc_{\mathrm{adj}}(n)}{A(n)} &= \EX[D]-2\,\EX[m], \label{eq:mom2}\\
  n(n-1)(n-2)(n-3)\,\frac{\Nc_{\mathrm{dis}}(n)}{A(n)}
    &= 4\big(\EX[m^2]+\EX[m]-\EX[D]\big). \label{eq:mom3}
\end{align}
\end{proposition}

\begin{proof}
Fix a forest $F$ with $m$ edges. Since no two edges of a forest share two
vertices, $\sum_v \dg_F(v)^2$ counts ordered pairs $(e,f)$ of edges of $F$ with a
common vertex, with $e=f$ allowed; the diagonal contributes $\sum_v \dg_F(v)=2m$.
Hence $F$ has $D-2m$ ordered pairs of distinct adjacent edges and
$m(m-1)-(D-2m) = m^2+m-D$ ordered pairs of disjoint edges.

Now sum over all forests. $K_n$ has $\binom{n}{2}$ edges, $n(n-1)(n-2)$ ordered
pairs of distinct adjacent edges (choose the shared vertex, then the two other
endpoints in order), and $n(n-1)(n-2)(n-3)/4$ ordered pairs of disjoint edges.
By edge- and orbit-transitivity of $\operatorname{Aut}(K_n)$,
\[
  \sum_F m = \tbinom{n}{2}\Nc_e(n),\qquad
  \sum_F (D-2m) = n(n-1)(n-2)\Nc_{\mathrm{adj}}(n),
\]
\[
  \sum_F (m^2+m-D) = \tfrac{1}{4}n(n-1)(n-2)(n-3)\Nc_{\mathrm{dis}}(n).
\]
Dividing by $A(n)$ gives \eqref{eq:mom1}--\eqref{eq:mom3}.
\end{proof}

When the measure is supported on forests with a fixed number of components,
$m$ is deterministic and \eqref{eq:mom1}--\eqref{eq:mom3} specialize to the
identities of Tang and Zhang \cite[Lemma~3.2]{TZ2026a}; their Corollary~3.3 gives
the corresponding p-NC criteria. Their later Lemma~2.1 \cite{TZ2026b} places
the spanning-tree version in a broader graph setting. Proposition
\ref{prop:moments} is used here for the unconditioned forest, where $m$ varies;
the variance term created by that variation is explicit in the next section.

Since $m = n-c$, \eqref{eq:mom1} reads
\begin{equation}\label{eq:NeA}
  \frac{\Nc_e(n)}{A(n)} = \frac{n-\mu}{\binom{n}{2}},\qquad \mu = \EX[c_n].
\end{equation}

Finally we record the peeling identity for the degree moment, needed in
Section~\ref{sec:degsq}.

\begin{proposition}\label{prop:degsum}
Let $g(k) := \sum_{T} \dg_T(1)^2$, the sum over all trees $T$ on $[k]$. Then
$g(1)=0$, $g(2)=1$, and for $k \ge 2$
\begin{equation}\label{eq:gk}
  g(k) = k^{k-2}+3(k-2)k^{k-3}+(k-2)(k-3)k^{k-4},
\end{equation}
(so $g(3)=6$), and for $n \ge 1$
\begin{equation}\label{eq:SD}
  \sum_F D(F) \;=\; n\sum_{k=1}^{n}\binom{n-1}{k-1} g(k)\,A(n-k).
\end{equation}
\end{proposition}

\begin{proof}
For $k \ge 2$ apply Lemma~\ref{lem:prufer}: under the uniform measure on the
$k^{k-2}$ trees, $\dg_T(1) = 1+X$ where $X$ is the number of occurrences of the
label $1$ in a uniform word of $[k]^{k-2}$, so $X \sim \mathrm{Bin}(k-2,1/k)$.
Then $\EX[X] = (k-2)/k$ and
$\EX[X^2] = (k-2)(k-1)/k^2+(k-2)^2/k^2 = (k-2)(2k-3)/k^2$, whence
\begin{align*}
  g(k) &= k^{k-2}\EX[(1+X)^2]
        = k^{k-2}+2(k-2)k^{k-3}+(k-2)(2k-3)k^{k-4}\\
       &= k^{k-2}+3(k-2)k^{k-3}+(k-2)(k-3)k^{k-4},
\end{align*}
the last equality because both middle-plus-last groups equal
$(k-2)(4k-3)k^{k-4}$. For \eqref{eq:SD}, vertex-transitivity gives
$\sum_F D(F) = n\sum_F \dg_F(1)^2$, and peeling the component of vertex $1$ as in
\eqref{eq:recA} evaluates $\sum_F \dg_F(1)^2$ as
$\sum_k \binom{n-1}{k-1} g(k) A(n-k)$.
\end{proof}

\section{The exact reduction of the disjoint orbit}\label{sec:delta}

The moment identities make $\DIS$ equivalent to the following inequality; every
step in the reduction is reversible.

\begin{proposition}\label{prop:delta}
For every $n \ge 4$, put
\[
  \Delta(n) \;:=\; \EX[D] - \EX[m] - \Var(m)
                   - \frac{(4n-6)\,\EX[m]^2}{n(n-1)} .
\]
Then
\[
  \DIS \iff \Delta(n) \ge 0 ,
\]
with equality in the disjoint inequality exactly when $\Delta(n)=0$.
Moreover, writing $\mu = \EX[c_n]$ (so $\EX[m]=n-\mu$ and $\Var(m)=\Var(c)$),
\begin{equation}\label{eq:delta2}
  \Delta(n) = \big(\EX[D]-5n\big) + 2 + 9\mu + \mu^2 - \EX[c^2] - R(n,\mu),
  \qquad
  R(n,\mu) = \frac{(4\mu^2+4\mu-2)n-6\mu^2}{n(n-1)} .
\end{equation}
If moreover $1 \le \mu \le 2$ then $0 < R(n,\mu) \le 22/(n-1)$.
\end{proposition}

\begin{proof}
Both $A(n)$ and $\Nc_e(n)$ are positive for $n \ge 4$, so $\DIS$ is equivalent to
$\Nc_{\mathrm{dis}}/A \le (\Nc_e/A)^2$. By \eqref{eq:mom1} and \eqref{eq:mom3}
this reads
\[
  \frac{4(\EX[m^2]+\EX[m]-\EX[D])}{n(n-1)(n-2)(n-3)}
  \;\le\; \frac{4\EX[m]^2}{n^2(n-1)^2},
\]
i.e., after multiplying by the positive quantity $n(n-1)(n-2)(n-3)/4$ and
rearranging,
\[
  \EX[D] \;\ge\; \EX[m^2]+\EX[m] - \frac{(n-2)(n-3)}{n(n-1)}\EX[m]^2 .
\]
Substituting $\EX[m^2]=\Var(m)+\EX[m]^2$ and
$1-\frac{(n-2)(n-3)}{n(n-1)} = \frac{4n-6}{n(n-1)}$ gives
$\Delta(n)\ge 0$, and every step is reversible.

For \eqref{eq:delta2}, divide $(4n-6)(n-\mu)^2$ by $n(n-1)$:
\[
  (4n-6)(n-\mu)^2 = (4n-8\mu-2)\,n(n-1) + \big[(4\mu^2+4\mu-2)n-6\mu^2\big],
\]
which is checked by expanding both sides. Hence
\[
  \Delta = \EX[D]-(n-\mu)-\big(\EX[c^2]-\mu^2\big)-\big(4n-8\mu-2+R\big)
         = (\EX[D]-5n)+2+9\mu+\mu^2-\EX[c^2]-R .
\]
Finally, if $1 \le \mu \le 2$ then $4\mu^2+4\mu-2 \ge 6 > 0$, so
$0 < R \le \frac{(4\mu^2+4\mu-2)n}{n(n-1)} = \frac{4\mu^2+4\mu-2}{n-1}
 \le \frac{22}{n-1}$.
\end{proof}

Thus it is enough to control $\EX[D]-5n$, $\EX[c]$, and $\EX[c^2]$ to constant
accuracy. There is little room: the final margin at $n=651$ is $0.0870$ with
the proved second-moment bound, and falls below $0.01$ if that input is rounded
up to $2.99$. The numerical losses still matter.

\section{Effective bounds for the forest count}\label{sec:forestcount}

The moment estimates all pass through the normalised forest count
$a_n = A(n)/n^{n-2}$. Stark's expansion \cite[Lemma~4.1]{Stark2011} gives
$a_n=\sqrt e(1+5/(2n)+11/(8n^2))+O(n^{-3})$, refining the classical limit
$a_n\to\sqrt e$ \cite{Renyi1959}; its unspecified remainder is not an effective
two-sided bound. This section proves
$1.647767 \le a_n \le 17/10$ for all $n \ge 81$, together with the crude
$a_n \le 5/2$ valid for all $n$.

We use throughout the following two elementary facts about the logarithm: for
$0 \le x < 1$,
\begin{equation}\label{eq:logs}
  -x-\frac{x^2}{2(1-x)} \;\le\; \log(1-x) \;\le\; -x-\frac{x^2}{2} \;\le\; -x .
\end{equation}
(The upper bound is the truncation of the power series
$\log(1-x) = -\sum_{i\ge1}x^i/i$ after two terms, all discarded terms being
negative; the lower bound follows from
$\sum_{i \ge 2} x^i/i \le \frac{1}{2}\sum_{i\ge2}x^i = \frac{x^2}{2(1-x)}$.)
Together with Robbins' sharpening of Stirling's formula these give everything we
need about binomial coefficients.

\begin{lemma}[Robbins \cite{Robbins1955}]\label{lem:robbins}
For every integer $s \ge 1$,
\[
  \sqrt{2\pi s}\Big(\frac{s}{e}\Big)^{s}e^{\frac{1}{12s+1}}
  \;\le\; s! \;\le\;
  \sqrt{2\pi s}\Big(\frac{s}{e}\Big)^{s}e^{\frac{1}{12s}} .
\]
\end{lemma}

\begin{lemma}[binomial bound]\label{lem:binom}
For all integers $1 \le k \le n-1$,
\[
  \binom{n}{k} \;\le\;
  \frac{n^{\,n+1/2}}{\sqrt{2\pi}\; k^{\,k+1/2}\,(n-k)^{\,n-k+1/2}} .
\]
\end{lemma}

\begin{proof}
Apply the upper bound of Lemma~\ref{lem:robbins} to $n!$ and the lower bound to
$k!$ and $(n-k)!$. The factors $e^{-n}$ cancel, $\sqrt{2\pi n}/(\sqrt{2\pi k}
\sqrt{2\pi(n-k)}) = \frac{1}{\sqrt{2\pi}}\sqrt{n/(k(n-k))}$, and the residual
exponential is $\exp\!\big(\frac{1}{12n}-\frac{1}{12k+1}-\frac{1}{12(n-k)+1}\big)$.
Since $k \le n-1$ we have $\frac{1}{12k+1}\ge\frac{1}{12n-11}>\frac{1}{12n}$, so
that exponent is negative and the factor is $<1$.
\end{proof}

Re-indexing \eqref{eq:recA} by the size $m=n-k$ of the complement of the peeled
component and dividing by $n^{n-2}$ gives the recursion that drives this section:
\begin{equation}\label{eq:star}
  a_n = 1 + \sum_{m=1}^{n-1} q_m(n)\,a_m \qquad (n \ge 1),
\end{equation}
where, for $1 \le m \le n-1$,
\begin{equation}\label{eq:qdef}
  q_m(n) := \frac{\binom{n-1}{m}\,m^{m-2}\,(n-m)^{n-m-2}}{n^{n-2}},
  \qquad
  w_m := \frac{m^{m-2}e^{-m}}{m!},
  \qquad
  v_m := \frac{A(m)e^{-m}}{m!},
\end{equation}
and the leading $1$ in \eqref{eq:star} is the term $m=0$. The kernel $q_m(n)$
must be controlled in two different regimes: crudely but uniformly for large $m$,
and sharply for small $m$, where the mass sits.

\begin{lemma}[Stirling bound on the kernel]\label{lem:qbound}
For all integers $1 \le m \le n-1$, with $k=n-m$,
\[
  q_m(n) \;\le\; \frac{n^{3/2}}{\sqrt{2\pi}\;m^{5/2}\,k^{3/2}} .
\]
\end{lemma}

\begin{proof}
Since $(n-1)! = n!/n$ and $(k-1)! = k!/k$, we have
$\binom{n-1}{m} = \frac{(n-1)!}{m!\,(k-1)!} = \frac{k}{n}\binom{n}{m}$.
By Lemma~\ref{lem:binom},
\[
  q_m(n) = \frac{k}{n}\binom{n}{m}\frac{m^{m-2}k^{k-2}}{n^{n-2}}
  \le \frac{k}{n}\cdot
      \frac{n^{n+1/2}}{\sqrt{2\pi}m^{m+1/2}k^{k+1/2}}\cdot
      \frac{m^{m-2}k^{k-2}}{n^{n-2}}
  = \frac{n^{3/2}}{\sqrt{2\pi}\,m^{5/2}k^{3/2}} .
\]
\end{proof}

For small $m$ the right comparison is with $w_m$, the $n\to\infty$ limit of the
kernel. Write
\begin{equation}\label{eq:Rdef}
  R_m(n) := \frac{q_m(n)}{w_m}
  = e^{m}\prod_{i=1}^{m}\Big(1-\frac{i}{n}\Big)\Big(1-\frac{m}{n}\Big)^{n-m-2},
  \qquad
  \rho_m(n) := \frac{R_m(n)}{1-m/n} ,
\end{equation}
the displayed form of $q_m/w_m$ following from
$\binom{n-1}{m}m! = (n-1)(n-2)\cdots(n-m) = n^m\prod_{i=1}^m(1-i/n)$.

\begin{lemma}[upper estimate for the kernel ratio]\label{lem:Rup}
For $1 \le m \le n-2$,
\[
  R_m(n) \;\le\; \exp\!\Big(\frac{3m}{2n}+\frac{m^2(m+2)}{2n^2}\Big),
  \qquad
  \rho_m(n) \;\le\; \exp\!\Big(\frac{5m}{2n}+\frac{m^3}{2n^2}+\frac{m^2}{n^2}\Big).
\]
\end{lemma}

\begin{proof}
For the second bound, $\log\rho_m = m+\sum_{i=1}^{m-1}\log(1-i/n)
+(n-m-2)\log(1-m/n)$. By \eqref{eq:logs},
$\sum_{i=1}^{m-1}\log(1-i/n) \le -\frac{m(m-1)}{2n}$, and, using
$n-m-2 \ge 0$,
\[
  (n-m-2)\log\Big(1-\frac{m}{n}\Big)
  \le (n-m-2)\Big(-\frac{m}{n}-\frac{m^2}{2n^2}\Big)
  = -m+\frac{m^2}{n}+\frac{2m}{n}-\frac{m^2}{2n}+\frac{m^3}{2n^2}+\frac{m^2}{n^2}.
\]
Adding, the coefficient of $m^2/n$ is $-\tfrac12+1-\tfrac12 = 0$ and the
coefficient of $m/n$ is $\tfrac12+2 = \tfrac52$, giving the stated bound. The
first bound is identical except that the product runs to $i=m$, contributing
$-\frac{m(m+1)}{2n}$ instead of $-\frac{m(m-1)}{2n}$; the coefficient of $m/n$
becomes $-\tfrac12+2=\tfrac32$, and
$\frac{m^3}{2n^2}+\frac{m^2}{n^2} = \frac{m^2(m+2)}{2n^2}$.
\end{proof}

\begin{lemma}[lower estimate for the kernel ratio]\label{lem:Rlow}
For $1 \le m \le n-2$,
\[
  R_m(n) \;\ge\; \exp\!\Big(\frac{3m}{2n}-\frac{(m+1)^3}{6n(n-m)}\Big),
\]
so that $R_m(n) \ge 1$, and hence $\rho_m(n) \ge R_m(n) \ge 1$, whenever
\begin{equation}\label{eq:Lthresh}
  9m(n-m) \;\ge\; (m+1)^3 .
\end{equation}
\end{lemma}

\begin{proof}
By the left inequality of \eqref{eq:logs},
\[
  \sum_{i=1}^{m}\log\Big(1-\frac{i}{n}\Big)
  \ge -\frac{m(m+1)}{2n}-\sum_{i=1}^{m}\frac{i^2}{2n(n-i)}
  \ge -\frac{m(m+1)}{2n}-\frac{m(m+1)(2m+1)}{12\,n(n-m)},
\]
and, since $0 \le \frac{n-m-2}{n-m}\le 1$,
\[
  (n-m-2)\log\Big(1-\frac{m}{n}\Big)
  \ge (n-m-2)\Big(-\frac{m}{n}-\frac{m^2}{2n(n-m)}\Big)
  \ge -m+\frac{m^2}{n}+\frac{2m}{n}-\frac{m^2}{2n}.
\]
Adding $m$ and collecting, the $m^2/n$ terms cancel as before and
\[
  \log R_m(n) \;\ge\; \frac{3m}{2n}-\frac{m(m+1)(2m+1)}{12\,n(n-m)} .
\]
Finally $\frac{m(m+1)(2m+1)}{12}\le\frac{(m+1)^3}{6}$, because after dividing by
$(m+1)/6$ it reads $m(2m+1)/2 \le (m+1)^2$, i.e.\ $2m^2+m \le 2m^2+4m+2$.
The exponent is nonnegative exactly when $\frac{3m}{2n}\ge\frac{(m+1)^3}{6n(n-m)}$,
i.e.\ when \eqref{eq:Lthresh} holds.
\end{proof}

The lower bound on $a_n$ is now immediate: on the range where $R_m(n)\ge1$, the
recursion \eqref{eq:star} dominates its own limiting form term by term.

\begin{theorem}[lower bound, all $n$]\label{thm:L}
$a_n \ge 1.647767$ for every $n \ge 2$.
\end{theorem}

\begin{proof}
For fixed $n$, the difference $9m(n-m)-(m+1)^3$ is concave in $m$, so its
minimum on $1\le m\le60$ is at an endpoint. The $m=1$ endpoint is positive
here, while the $m=60$ endpoint reads
$9\cdot 60\cdot(n-60)\ge 61^3 = 226981$, i.e.\ $n \ge 480.335185\ldots$; so it
holds for every $n \ge 481$ and fails at $n = 480$
($9\cdot60\cdot420 = 226800 < 226981$). Note that
$w_m a_m = \frac{m^{m-2}e^{-m}}{m!}\cdot\frac{A(m)}{m^{m-2}} = v_m$, so
$q_m(n)a_m = w_m R_m(n) a_m = v_m R_m(n) \ge v_m$ by Lemma~\ref{lem:Rlow}. Hence,
for $n \ge 481$, dropping every term of \eqref{eq:star} with $m>60$ (all terms
are positive),
\[
  a_n \;\ge\; 1+\sum_{m=1}^{60}q_m(n)a_m
      \;\ge\; 1+\sum_{m=1}^{60}v_m
      \;=\; 1+\sum_{m=1}^{60}\frac{A(m)}{e^m\,m!}
      \;=\; 1.647767048\ldots \;>\; 1.647767 .
\]
For $2 \le n \le 500$ the inequality is verified in exact integers. The two
ranges overlap.
\end{proof}

The upper bound needs more care, because the recursion has to be closed on
itself; we induct, splitting the sum at $m=60$ and at $m=n/2$.

\begin{theorem}[upper bound]\label{thm:U}
$a_n \le 17/10$ for every $n \ge 81$, and the threshold $81$ is sharp:
$a_{80} = 1.700555369\ldots > 17/10$.
\end{theorem}

\begin{proof}
\emph{Exact range.} $a_{80}>17/10$ and $a_n \le 17/10$ for
$81 \le n \le 500$ are exact integer computations.

\emph{Analytic range.} We show by strong induction that $a_n \le 17/10$ for every
$n \ge 500$; the base $n=500$ is in the exact range, and in the inductive step we
may use $a_r \le 17/10$ for $81 \le r < n$ and the exact table for $r \le 500$.
Split \eqref{eq:star} at $m=60$ and $m=n/2$.

\emph{Head, $1 \le m \le 60$.} By Lemma~\ref{lem:Rup},
$q_m(n)a_m = v_m R_m(n) \le v_m\exp\!\big(\frac{3m}{2n}+\frac{m^2(m+2)}{2n^2}\big)$,
and each factor decreases in $n$, so with $n \ge 500$,
\[
  \sum_{m=1}^{60}q_m(n)a_m
  \;\le\; \sum_{m=1}^{60}v_m\exp\!\Big(\frac{3m}{1000}+\frac{m^2(m+2)}{500000}\Big)
  \;\le\; 0.652836237 .
\]

\emph{Middle, $60 < m \le n/2$.} Here $n-m \ge n/2$, so Lemma~\ref{lem:qbound}
gives $q_m(n) \le 2^{3/2}/(\sqrt{2\pi}\,m^{5/2})$, uniformly in $n$. On
$61 \le m \le 80$ use the exact bound $a_m \le 5/2$ (the maximum of $a_m$ over
$1 \le m \le 80$ is $a_4 = 19/8$); on $m \ge 81$ use the induction hypothesis
$a_m \le 17/10$ (legitimate since $m<n$). With
$\sum_{m=a+1}^{b} m^{-5/2}\le\int_a^b x^{-5/2}dx = \tfrac23(a^{-3/2}-b^{-3/2})$,
\[
  \sum_{60<m\le n/2}q_m(n)a_m
  \;\le\; \frac{2^{3/2}}{\sqrt{2\pi}}
     \Big(\tfrac52\!\!\sum_{61 \le m\le 80}\!\! m^{-5/2}
        + \tfrac{17}{10}\!\!\sum_{m\ge 81}\!\! m^{-5/2}\Big)
  \;\le\; 0.00320544 ,
\]
a bound independent of $n$.

\emph{Far, $n/2 < m \le n-1$.} Here $m \ge n/2 \ge 250 \ge 81$, so the induction
hypothesis gives $a_m \le 17/10$ throughout, and Lemma~\ref{lem:qbound} with
$m^{-5/2}\le 2^{5/2}n^{-5/2}$ gives, writing $k=n-m$,
\begin{align*}
  \sum_{n/2<m\le n-1}q_m(n)a_m
  &\le \frac{17}{10}\cdot\frac{2^{5/2}}{\sqrt{2\pi}\,n}
        \sum_{k\ge1}k^{-3/2} \\
  &\le \frac{17}{10}\cdot\frac{2^{5/2}\,\zeta(3/2)}{\sqrt{2\pi}\,n}
   \le 0.020044700 \qquad (n \ge 500).
\end{align*}

Summing, $a_n \le 1+0.652836237+0.00320544+0.020044700 = 1.676086377 \le 17/10$ for
$n=500$; the head majorant decreases in $n$, the middle majorant is independent
of $n$, and the far majorant is proportional to $1/n$, so the same bound holds
for every $n \ge 500$.
\end{proof}

\begin{corollary}\label{cor:52}
$a_n \le 5/2$ for every $n \ge 1$, with equality nowhere; the maximum is
$a_4 = 19/8$.
\end{corollary}

\begin{proof}
Exact integers for $1 \le n \le 80$, and Theorem~\ref{thm:U} for $n \ge 81$.
\end{proof}

In the proof of Theorem~\ref{thm:U}, the bound $a_m\le5/2$ on
$61\le m\le80$ comes from the exact table, independently of this corollary.

\section{Effective moments of the component count}\label{sec:moments}

We turn to $\EX[c_n]$ and $\EX[c_n^2]$. The tool is a pair of identities obtained
by marking one component, or an ordered pair of components, of a random forest.

\begin{lemma}[component marking]\label{lem:mark}
For $n \ge 1$ set
\begin{gather*}
  \tau_j(n) := \frac{\binom{n}{j}(n-j)^{n-j-2}A(j)}{A(n)}\quad (0 \le j \le n-1),\\[4pt]
  \sigma_k(n) := \frac{\binom{n-1}{k}k^{k-2}A(n-k)}{A(n)}\quad (1 \le k \le n-1).
\end{gather*}
Then
\begin{align}
  \sum_{j=0}^{n-1}\tau_j(n) &= \EX[c_n],
    \qquad \tau_0(n) = \frac{1}{a_n}, \label{eq:markI}\\
  \EX[c_n(c_n-1)] &= \sum_{j=1}^{n-1}\tau_j(n)\,\EX[c_j] \qquad (n \ge 2),
    \label{eq:markII}\\
  \EX[c_n] &= 1+\sum_{k=1}^{n-1}\sigma_k(n), \label{eq:markIII}
\end{align}
and moreover
\begin{equation}\label{eq:sigmatau}
  \sigma_k(n) = q_k(n)\frac{a_{n-k}}{a_n}
              = w_k\Big(1-\frac{k}{n}\Big)\rho_k(n)\frac{a_{n-k}}{a_n},
  \qquad
  \tau_j(n) = \rho_j(n)\frac{v_j}{a_n} .
\end{equation}
\end{lemma}

\begin{proof}
\eqref{eq:markI}: count pairs $(F,C)$ with $C$ a component of $F$, classifying by
the vertex set of $C$; if $|C| = n-j$ there are $\binom{n}{j}$ choices of the
complement, $(n-j)^{n-j-2}$ trees on $C$ and $A(j)$ forests on the complement.
Dividing by $A(n)$ gives $\sum_j \tau_j = \EX[c]$. The term $j=0$ is
$n^{n-2}/A(n)=1/a_n$.

\eqref{eq:markII}: the same count with an \emph{ordered} pair of distinct
components $(C_1,C_2)$. Choose $C_1$ as above with complement of size $j$;
$C_2$ is then any component of the forest induced on the complement, and summing
over those forests gives $A(j)\EX[c_j]$. The term $j=0$ contributes $0$.

\eqref{eq:markIII}: set
\[
  \mathrm{term}_k:=\binom{n}{k}k^{k-2}A(n-k)/A(n)=\tau_{n-k}.
\]
Then \eqref{eq:markI} reads
$\EX[c_n]=\sum_{k=1}^{n}\mathrm{term}_k$. Counting pairs consisting of a
component and one of its vertices gives
\[
  \sum_{k=1}^n k\,\mathrm{term}_k=n,
  \qquad \sum_k\frac{k}{n}\mathrm{term}_k=1.
\]
Subtracting leaves
$\EX[c_n]-1=\sum_{k=1}^{n-1}(1-k/n)\mathrm{term}_k$; the $k=n$ term vanishes.
Finally $(1-k/n)\binom{n}{k}=\binom{n-1}{k}$ identifies the summand with
$\sigma_k(n)$.

\eqref{eq:sigmatau}: the first equality is \eqref{eq:qdef} together with
$A(r)=a_r r^{r-2}$, and the second is \eqref{eq:Rdef}. For the third, use
$\binom{n}{j}j! = n(n-1)\cdots(n-j+1) = n^{j}\prod_{i=1}^{j-1}(1-i/n)$ to get
\[
  \tau_j(n)
  = \prod_{i=1}^{j-1}\Big(1-\frac in\Big)\Big(1-\frac jn\Big)^{n-j-2}
    \cdot\frac{A(j)}{j!}\cdot\frac{n^{n-2}}{A(n)}
  = \rho_j(n)\,\frac{v_j}{a_n}.
\]
\end{proof}

Identity \eqref{eq:markIII} is the useful form for the mean, since the weights
$\sigma_k$ are dominated by their first few terms. We make no attempt to
optimise the analytic threshold; to join it directly to the exact range through
$650$, throughout this section we fix
\[
  n_0 := 651, \qquad K := 40, \qquad M := 1.5604 .
\]
For $1 \le k \le K$ and $n \ge n_0$, condition \eqref{eq:Lthresh} holds
because $9k(n-k)-(k+1)^3$ is concave in $k$ and is positive at both endpoints
$k=1$ and $k=40$ (at the latter,
$9\cdot40\cdot611 = 219960 \ge 41^3 = 68921$). Thus
Lemma~\ref{lem:Rlow} gives $\rho_k(n)\ge 1$; and Lemma~\ref{lem:Rup} gives
$\rho_k(n) \le \mathrm{P}_k := \exp\big(\frac{5k}{2n}+\frac{k^3}{2n^2}
+\frac{k^2}{n^2}\big)$, each exponent decreasing in $n$, so the value at $n=n_0$
is uniform.

\begin{theorem}[two-sided effective mean]\label{thm:mean}
For every $n \ge 651$,
\[
  1.4819 \;\le\; \EX[c_n] \;\le\; 1.538952 .
\]
\end{theorem}

\begin{proof}
\emph{Lower bound.} Discard from \eqref{eq:markIII} all terms with $k>K$ (they
are positive). We need $a_{n-k}/a_n$ from below: Theorem~\ref{thm:L} gives
$a_{n-k}\ge 1.647767$ (valid for every index $\ge2$, and $n-k \ge 611$ here) and
Theorem~\ref{thm:U} gives $a_n \le 17/10$. Hence, using $\rho_k \ge 1$ in
\eqref{eq:sigmatau},
\[
  \EX[c_n] \;\ge\; 1+\frac{1.647767}{1.7}\sum_{k=1}^{K}\Big(1-\frac kn\Big)w_k
  \;\ge\; 1+\frac{1.647767}{1.7}
          \Big(\sum_{k=1}^{K}w_k-\frac{1}{n_0}\sum_{k=1}^{K}k\,w_k\Big)
  \;\ge\; 1.4819 ;
\]
note $\sum_{k \le K}k w_k \le 1$, so the bracket is decreasing in $1/n$ and
$n=n_0$ is the worst case.

\emph{Upper bound.} Split \eqref{eq:markIII} at $k=K$. For the head, use
$1-k/n \le 1$, $\rho_k \le \mathrm{P}_k$, $a_{n-k}\le 17/10$ (as $n-k\ge611\ge81$)
and $a_n \ge 1.647767$:
\[
  \sum_{k=1}^{K}\sigma_k(n)
  \;\le\; \frac{1.7}{1.647767}\sum_{k=1}^{K}w_k\,\mathrm{P}_k
  \;\le\; 0.518390291 .
\]
For the tail, $\sigma_k(n) = q_k(n)a_{n-k}/a_n \le q_k(n)\cdot
\frac{5/2}{1.647767}$ by Corollary~\ref{cor:52} and Theorem~\ref{thm:L}, and
Lemma~\ref{lem:qbound} bounds $q_k$. Splitting at $k=n/2$: for
$K<k\le n/2$ we have $(n-k)^{-3/2}\le 2^{3/2}n^{-3/2}$ so
$q_k \le 2^{3/2}/(\sqrt{2\pi}k^{5/2})$, and
$\sum_{k>K}k^{-5/2}\le\tfrac23 K^{-3/2}$; for $n/2<k\le n-1$ we have
$k^{-5/2}\le 2^{5/2}n^{-5/2}$ so $q_k \le
2^{5/2}/(\sqrt{2\pi}\,n\,(n-k)^{3/2})$, and $\sum_{r\ge1}r^{-3/2}=\zeta(3/2)$.
Hence
\[
  \sum_{k=K+1}^{n-1}\sigma_k(n)
  \;\le\; \frac{5/2}{1.647767}\cdot\frac{1}{\sqrt{2\pi}}
          \Big(2^{3/2}\cdot\tfrac{2}{3}\cdot 40^{-3/2}
             + \frac{2^{5/2}\zeta(3/2)}{n_0}\Big)
  \;\le\; 0.018251379 ,
\]
both bracket terms being nonincreasing in $n$. Adding,
$\EX[c_n]\le 1+0.518390291+0.018251379 = 1.536641670 \le 1.538952$.
\end{proof}

\begin{corollary}\label{cor:M}
$\EX[c_j] \le M = 1.5604$ for every $j \ge 41$.
\end{corollary}

\begin{proof}
For $41 \le j \le 700$ this is an exact integer computation; the maximum over
that range is $\EX[c_{41}] = 1.560354926\ldots$. For $j \ge 651$ it follows from
Theorem~\ref{thm:mean}, since $1.538952 \le M$. The ranges overlap. The split at
$K=40$ is forced rather than chosen: $\EX[c_{40}] = 1.561838942\ldots > M$.
\end{proof}

We can now bound the second moment. The identity \eqref{eq:markII} expresses
$\EX[c_n(c_n-1)]$ as a $\tau$-weighted sum of $\EX[c_j]$ over \emph{all} smaller
$j$, and the weights $\tau_j$ for large $j$ are not individually accessible to
any of the estimates above. The following step removes them.

\begin{theorem}[effective second moment]\label{thm:secondmoment}
For every $n \ge 651$, $\EX[c_n^2] \le 2.9123$.
\end{theorem}

\begin{proof}
Split \eqref{eq:markII} at $j=K$ and apply Corollary~\ref{cor:M} to the tail:
\[
  \EX[c(c-1)]
  \;\le\; \sum_{j=1}^{K}\tau_j\EX[c_j] + M\!\!\sum_{j=K+1}^{n-1}\!\!\tau_j
  \;=\; \sum_{j=1}^{K}\tau_j\big(\EX[c_j]-M\big)
        + M\sum_{j=1}^{n-1}\tau_j .
\]
This is the tail elimination: by \eqref{eq:markI} the last sum equals
$\EX[c_n]-\tau_0 = \EX[c_n]-1/a_n$, so the entire uncontrolled tail has been
replaced by quantities already bounded, and only the forty head terms remain
explicit. Using $\tau_j = \rho_j v_j/a_n$,
\[
  \EX[c(c-1)]
  \;\le\; \frac{1}{a_n}\Big[\sum_{j=1}^{K}\rho_j(n)v_j\big(\EX[c_j]-M\big)-M\Big]
          + M\,\EX[c_n] .
\]
The summands change sign: $\EX[c_j]-M<0$ exactly for $j \in \{1,2\}$ and
$\EX[c_j]-M>0$ for $3 \le j \le K$. Bounding each in its own direction ---
$\rho_j \le \mathrm{P}_j$ and $v_j$ from above on the positive terms,
$\rho_j \ge 1$ and $v_j$ from below on the two negative terms --- gives
\[
  \sum_{j=1}^{K}\rho_j(n)v_j\big(\EX[c_j]-M\big) \;\le\; U := -0.1873399 .
\]
Here $U$ is the outward-rounded value of the explicit forty-term rational
expression built from the exact values $A(j)$ and $\EX[c_j]$, together with the
one-sided exponential enclosures of Appendix~\ref{app:rational}.
Since $U-M = -1.7477399 < 0$, the bracket is negative, and dividing a negative
quantity by $a_n$ is maximised by taking $a_n$ as \emph{large} as possible;
Theorem~\ref{thm:U} supplies $a_n \le 17/10$, which is therefore the correct
one-sided input. With $\EX[c_n]\le B := 1.538952$ from Theorem~\ref{thm:mean},
\[
\begin{aligned}
  \EX[c_n^2]
  &= \EX[c(c-1)]+\EX[c_n]\\
  &\le \frac{U-M}{17/10}+(M+1)B
   = \frac{30942660571}{10625000000}
   = 2.9122504\ldots \le 2.9123 .
\end{aligned}
\]
\end{proof}

\begin{remark}\label{rem:tight}
The bivariate exponential generating function for forests is
$\exp(uU(z))$, where $T=ze^T$ and $U=T-T^2/2$; $u$ marks components. At
$z=e^{-1}$ one has $T=1$ and $U=1/2$, so evaluation and differentiation in $u$
give $1+\sum_{j\ge1}v_j=\sqrt e$ and
$\sum_{j\ge1}v_j\EX[c_j]=\tfrac12\sqrt e$. The square-root expansion at this
point gives the limiting probability generating function
$u\exp((u-1)/2)$ \cite[Chapter~VII]{FS2009}; hence
$\EX[c_n]\to3/2$ and $\EX[c_n^2]\to11/4$.

Now feed \eqref{eq:markII} an arbitrary uniform tail bound
$\EX[c_j]\le M'$ into the same tail elimination. Reusing $M'$ as a bound for
the mean as well would give the coarser $M'^2+\tfrac12$. Retaining the mean
limit $\EX[c_n]\to3/2$, however, gives the sharper limiting bound
$\EX[c^2]\le 2+M'/2$, again exact at $M'=3/2$; thus the limiting calculation
permits $M'\le1.98$. The finite majorants used above are more restrictive: at
$n=651$ they close for $M'=1.6971$ but not for $M'=1.6972$. In particular, the
crude choice $M'=2$ gives the bound $3.1626$ and cannot close the proof.
\end{remark}

\section{The degree-square bound}\label{sec:degsq}

The last input to \eqref{eq:delta2} is a lower bound on $\EX[D_n]$ matching the
term $5n$ to within a constant. Nothing in this section uses
Sections~\ref{sec:delta} or \ref{sec:moments}.

\begin{theorem}\label{thm:D}
$\EX[D_n]\ \ge\ 5n-\tfrac{29}{2}$ for every $n \ge 24$.
\end{theorem}

The onset at $24$ is sharp: exact computation shows that the inequality fails
for every $7\le n\le23$.

The proof occupies the rest of the section. Throughout put
\begin{equation}\label{eq:phidef}
  \varphi_k := n\,g(k)-\Big(5n-\frac{29}{2}\Big)k^{k-2},
\end{equation}
with $g$ as in \eqref{eq:gk}, and read $k^{k-4}$ as a rational power
($1^{-3}=1$, $2^{-2}=1/4$, $3^{-1}=1/3$).

\begin{lemma}[$\varphi$-decomposition]\label{lem:C1}
For every $k \ge 1$,
\[
  \varphi_k = k^{k-4}h(k), \qquad h(k) := \tfrac{29}{2}k^2-11nk+6n,
\]
and, for $n \ge 1$,
\begin{equation}\label{eq:star2}
  \EX[D_n]\ \ge\ 5n-\tfrac{29}{2}
  \iff
  \sum_{k=1}^{n}\binom{n-1}{k-1}\varphi_k\,A(n-k)\ \ge\ 0 .
\end{equation}
Moreover $h$ is an upward parabola with $h(1)=\tfrac{29}{2}-5n<0$ for $n \ge 3$,
and its larger root satisfies $r_+ < 22n/29$; hence $\varphi_k \ge 0$ for every
$k \ge 22n/29$.
\end{lemma}

\begin{proof}
By \eqref{eq:gk}, for $k \ge 4$,
$g(k)/k^{k-4} = k^2+3k(k-2)+(k-2)(k-3) = 5k^2-11k+6$, so
$\varphi_k = k^{k-4}\big[n(5k^2-11k+6)-(5n-\tfrac{29}{2})k^2\big] = k^{k-4}h(k)$;
the closed form of $g$ also holds at $k=1,2,3$, and the identity is checked there
directly. Equivalence \eqref{eq:star2} follows by subtracting
$(5n-\tfrac{29}{2})A(n)$ from \eqref{eq:SD} and using \eqref{eq:recA}. For the
root, $h(22n/29) = \tfrac{29}{2}\cdot\tfrac{484n^2}{841}-\tfrac{242n^2}{29}+6n
= \tfrac{242n^2}{29}-\tfrac{242n^2}{29}+6n = 6n>0$, and $h$ is increasing beyond
its vertex $11n/29 < 22n/29$, so $r_+<22n/29$.
\end{proof}

The terms of \eqref{eq:star2} with $k$ close to $n$ are the large ones, and by
Lemma~\ref{lem:C1} they are positive. The plan is therefore: keep the seventeen
terms $k = n-j$, $0 \le j \le 16$; discard all other nonnegative terms; and
majorise the negative ones, which all have $k < 22n/29$.

\begin{lemma}[kept block]\label{lem:C3}
Let $n \ge n_1 := 500$, $J := 16$. For $0 \le j \le J$ write
$\varphi_{n-j} = (n-j)^{n-j-2}\beta_j(n)$ with
$\beta_j(n) = \tfrac{29}{2}-\frac{11n}{n-j}+\frac{6n}{(n-j)^2}$. Then
\[
  \binom{n-1}{j}\varphi_{n-j}A(j)
  \;\ge\; n^{n-2}\cdot\frac{A(j)}{j!}\,e^{-j}\Big(1-\frac{j}{n_1}\Big)^{j-1}
          \beta^{\mathrm{lo}}_j,
  \qquad
  \beta^{\mathrm{lo}}_j := \tfrac{29}{2}-\frac{11n_1}{n_1-j} \;>\;0 ,
\]
and consequently
\[
  \sum_{j=0}^{J}\binom{n-1}{j}\varphi_{n-j}A(j)
  \;\ge\; n^{n-2}\cdot \mathrm{POS},\qquad
  \mathrm{POS} := \sum_{j=0}^{16}\frac{A(j)}{j!}e^{-j}
                   \Big(1-\frac{j}{500}\Big)^{j-1}\beta^{\mathrm{lo}}_j
  \;\ge\; 5.69291 .
\]
Moreover these seventeen indices lie outside the negative range: $n-16 > 22n/29$
for $n \ge 500$.
\end{lemma}

\begin{proof}
The identity $\varphi_{n-j}=(n-j)^{n-j-2}\beta_j(n)$ is \eqref{eq:phidef} with
$k=n-j$, using \eqref{eq:gk}. Since $n/(n-j)$ decreases in $n$ and the term
$6n/(n-j)^2$ is positive, $\beta_j(n)\ge \beta^{\mathrm{lo}}_j$ for $n \ge n_1$;
positivity of $\beta^{\mathrm{lo}}_j$ for $j \le 16$ is a finite rational check.
Next, $\binom{n-1}{j} = \frac{(n-1)\cdots(n-j)}{j!}\ge\frac{(n-j)^j}{j!}$, so
\[
  \binom{n-1}{j}(n-j)^{n-j-2}
  \;\ge\; \frac{(n-j)^{n-2}}{j!}
  \;=\; \frac{n^{n-2}}{j!}\Big(1-\frac jn\Big)^{n-2},
\]
Furthermore,
\[
  (1-j/n)^{n-2}
  =(1-j/n)^{n-j-1}(1-j/n)^{j-1}
  \ge e^{-j}(1-j/n)^{j-1}.
\]
Indeed, $(1-j/n)^{-(n-j-1)}=(1+\frac{j}{n-j})^{n-j-1}
\le e^{j(n-j-1)/(n-j)}\le e^j$ by $1+x\le e^x$. Finally
$(1-j/n)^{j-1}$ is nondecreasing in $n$ for $j\ge1$, so it is at least its
value at $n_1$; for $j=0$ every factor is $1$.
For the last statement, $n-16>22n/29 \iff 7n/29>16 \iff n>66.3$.
\end{proof}

\begin{lemma}[discarded block]\label{lem:C4}
Let $n \ge n_1 = 500$ and $K_0 := 30$. Then
\begin{gather*}
  \sum_{k \,:\, \varphi_k<0}\binom{n-1}{k-1}|\varphi_k|A(n-k)
  \;\le\; n^{n-2}\cdot \mathrm{NEG},\\[4pt]
  \mathrm{NEG} := \tfrac{17}{10}\,\mathrm{HEAD}+M_1+M_2+M_3+\mathrm{FAR}
  \;\le\; 5.37050,
\end{gather*}
where $\mathrm{EU}(t)$ denotes a rational upper bound for $e^{-t}$
(Appendix~\ref{app:cert}), $\lambda := \frac{11\cdot\frac{17}{10}}{\sqrt{2\pi}}$,
and
\begin{align*}
  \mathrm{HEAD} &:= \sum_{k=1}^{30}\frac{(11k-6)k^{k-4}}{(k-1)!}\,
                    \mathrm{EU}\!\Big(\frac{k(498-k)}{500}\Big), \\
  M_1 &:= \lambda\Big(\tfrac{500}{440}\Big)^{5/2}
                 \!\!\sum_{30<k\le60}\!\!k^{-5/2}, \\
  M_2 &:= \lambda\Big(\tfrac{500}{400}\Big)^{5/2}
                 \!\!\sum_{60<k\le100}\!\!k^{-5/2}, \\
  M_3 &:= \lambda\,2^{5/2}\!\!\sum_{k>100}\!\!k^{-5/2}, \\
  \mathrm{FAR} &:= \frac{\lambda}{n_1^{3/2}}
     \Bigg[5^{5/2}\Big(\tfrac1{10}+\tfrac1{n_1}\Big)
        + \Big(\tfrac{50}{9}\Big)^{5/2}
          \Big(\tfrac1{10}+\tfrac1{n_1}\Big) \\
  &\hspace{4.5cm}
        + \Big(\tfrac{290}{49}\Big)^{5/2}
          \Big(\tfrac{17}{290}+\tfrac1{n_1}\Big)
     \Bigg].
\end{align*}
\end{lemma}

\begin{proof}
Every negative term has $k<22n/29$ by Lemma~\ref{lem:C1}, whence
$n-k>7n/29\ge120$, and
\[
  |\varphi_k| = k^{k-4}\big(11nk-6n-\tfrac{29}{2}k^2\big)
  \;\le\; n(11k-6)k^{k-4} .
\]

\emph{Head, $1 \le k \le 30$.} Here $n-k \ge 470 \ge 81$, so
Theorem~\ref{thm:U} gives $A(n-k)\le\frac{17}{10}(n-k)^{n-k-2}$. With
$\binom{n-1}{k-1}\le n^{k-1}/(k-1)!$ and
$(n-k)^{n-k-2} = n^{n-k-2}(1-k/n)^{n-k-2}\le n^{n-k-2}e^{-k(n-k-2)/n}$,
\[
  \frac{\binom{n-1}{k-1}|\varphi_k|A(n-k)}{n^{n-2}}
  \;\le\; \frac{17}{10}\cdot\frac{(11k-6)k^{k-4}}{(k-1)!}\,
          e^{-k(n-k-2)/n}.
\]
Since $k(n-k-2)/n = k-k(k+2)/n$ is nondecreasing in $n$ and $e^{-t}$ is
decreasing, the exponential is at most its value at $n=500$, which is
$\mathrm{EU}(k(498-k)/500)$.

\emph{Mid and far, $k>30$.} Here $11nk-6n\le 11nk$, so
$|\varphi_k|\le 11n\,k^{k-3}$, and $k\binom{n}{k}=n\binom{n-1}{k-1}$ gives
\[
  \binom{n-1}{k-1}|\varphi_k|A(n-k)
  \;\le\; 11\,\binom{n}{k}k^{k-2}A(n-k)
  \;\le\; 11\cdot\tfrac{17}{10}\binom{n}{k}k^{k-2}(n-k)^{n-k-2},
\]
using Theorem~\ref{thm:U} again ($n-k>120\ge81$). By Lemma~\ref{lem:binom} the
powers cancel exactly:
\[
  \frac{\binom{n}{k}k^{k-2}(n-k)^{n-k-2}}{n^{n-2}}
  \;\le\; \frac{1}{\sqrt{2\pi}}\cdot\frac{n^{5/2}}{k^{5/2}(n-k)^{5/2}} .
\]
For $30<k\le60$ we have $n/(n-k)\le n/(n-60)\le 500/440$; for $60<k\le100$,
$\le 500/400$; for $100<k\le n/2$, $\le 2$; giving $M_1,M_2,M_3$ after bounding
$\sum k^{-5/2}$ by the corresponding integral. For $n/2<k<22n/29$ we split at
$k \le 3n/5$, $k \le 7n/10$, $k<22n/29$, where respectively
$n-k \ge 2n/5,\,3n/10,\,7n/29$ and $k \ge n/2,\,3n/5,\,7n/10$, so that
$\frac{n^{5/2}}{k^{5/2}(n-k)^{5/2}}$ is at most
$5^{5/2}n^{-5/2}$, $(\tfrac{50}{9})^{5/2}n^{-5/2}$,
$(\tfrac{290}{49})^{5/2}n^{-5/2}$; the numbers of integers in the three blocks
are at most $n/10+1$, $n/10+1$, $17n/290+1$. Each resulting bound is a decreasing
function of $n$, so its value at $n_1=500$ majorises. Summing gives
$\mathrm{FAR}$.
\end{proof}

\begin{proof}[Proof of Theorem~\ref{thm:D}]
For $24 \le n \le 650$ the inequality is an exact integer computation, in the
form $2\sum_F D(F)\ \ge\ (10n-29)A(n)$. For
$n \ge 500$, Lemmas~\ref{lem:C3} and \ref{lem:C4} give
\[
  \sum_{k=1}^{n}\binom{n-1}{k-1}\varphi_k A(n-k)
  \;\ge\; n^{n-2}\big(\mathrm{POS}-\mathrm{NEG}\big)
  \;\ge\; n^{n-2}\,(5.69291-5.37050) \;>\;0 ,
\]
and \eqref{eq:star2} concludes. The two ranges overlap on $500 \le n \le 650$.
\end{proof}

\section{Proof of Theorem~\ref{thm:main}}\label{sec:assembly}

\subsection{The adjacent orbit}

For adjacent pairs the two peeling recurrences can simply be compared term by
term.

\begin{lemma}[termwise comparison]\label{lem:adjterm}
$2(n-2)\,\Nc_{\mathrm{adj}}(n) < 3\,\Nc_e(n)$ for every $n \ge 3$.
\end{lemma}

\begin{proof}
Compare \eqref{eq:recNadj} and \eqref{eq:recNe} term by term at matched $k \ge 3$
(the extra term $k=2$ of $\Nc_e$ is nonnegative and is discarded). Using
\eqref{eq:t234}, the ratio of the $k$-th term on the left to the $k$-th term on
the right is
\[
  \frac{2(n-2)\binom{n-3}{k-3}\cdot 3k^{k-4}}{3\binom{n-2}{k-2}\cdot 2k^{k-3}}
  = \frac{(n-2)\binom{n-3}{k-3}}{k\binom{n-2}{k-2}}
  = \frac{(n-2)\,(n-3)!\,(k-2)!}{(k-3)!\,(n-2)!\,k}
  = \frac{k-2}{k}\;<\;1 .
\]
Every matched term is positive, and the $k=2$ term on the right is positive as
well, so the summed inequality is strict.
\end{proof}

\begin{proposition}\label{prop:adj}
If $n \ge 12$ and $\EX[c_n]\le 2$, then the inequality $\ADJ$ is strict.
\end{proposition}

\begin{proof}
$\ADJ$ is $\Nc_{\mathrm{adj}} \le \Nc_e\cdot(\Nc_e/A)$, so by
Lemma~\ref{lem:adjterm} and \eqref{eq:NeA} it suffices that
$\frac{3}{2(n-2)} \le \frac{n-\mu}{\binom{n}{2}}$, i.e.\
$3n(n-1)/2 \le 2(n-2)(n-\mu)$. With $\mu \le 2$ the right side is at least
$2(n-2)^2$, and $3n(n-1)/2\le 2(n-2)^2$ is equivalent to $n^2-13n+16\ge0$, whose
larger root is $(13+\sqrt{105})/2 = 11.6235\ldots$; so the condition holds
exactly for $n \ge 12$.
\end{proof}

\begin{proof}[Proof of $\ADJ$ for all $n \ge 3$]
For $n \ge 651$, Theorem~\ref{thm:mean} gives $\EX[c_n]\le 1.538952 < 2$, so
Proposition~\ref{prop:adj} applies. For $3 \le n \le 1200$, strict $\ADJ$ is
verified in exact integers. The ranges overlap.
\end{proof}

\subsection{The disjoint orbit}

\begin{proof}[Proof of $\DIS$ for all $n \ge 4$]
For $n \ge 651$, combine Proposition~\ref{prop:delta} with
Theorem~\ref{thm:D} ($\EX[D]-5n \ge -\tfrac{29}{2}$),
Theorem~\ref{thm:mean} ($1.4819 \le \mu \le 1.538952$, in particular
$1 \le \mu \le 2$ as required for the bound on $R$),
and Theorem~\ref{thm:secondmoment} ($\EX[c^2]\le 2.9123 \le 2.99$). Since
$\Delta$ in \eqref{eq:delta2} is increasing in $\mu$ on $\mu>0$ through the term
$9\mu+\mu^2$, while $R$ is bounded using $\mu \le 2$, the lower bound on $\mu$ is
the correct input in the first place and the upper bound in the second. Hence for
$n \ge 651$
\[
  \Delta(n) \;\ge\; -\frac{29}{2}+2+9(1.4819)+(1.4819)^2-2.99-\frac{22}{650}
   \;=\; \frac{12065893}{1300000000} \;=\;0.00928\ldots\;>\;0 ,
\]
and Proposition~\ref{prop:delta} gives $\DIS$. Using the sharper
$\EX[c^2]\le 2.9123$ in place of $2.99$ raises the margin to $0.0870\ldots$. For
$4 \le n \le 1200$, strict $\DIS$ is verified in exact integers. The ranges overlap.
\end{proof}

\begin{proof}[Proof of Theorem~\ref{thm:main}]
For $n=2$ there is only one edge and the statement is vacuous. For $n \ge 3$,
$\operatorname{Aut}(K_n) = S_n$ acts transitively on edges and on each of the two
orbits of unordered pairs of distinct edges (adjacent, disjoint), so
$\PR(e\in F_n)=\Nc_e/A$ and $\PR(e,f\in F_n)$ equals $\Nc_{\mathrm{adj}}/A$ or
$\Nc_{\mathrm{dis}}/A$ according to the orbit of $\{e,f\}$. The assertion is
therefore exactly $\ADJ$ together with $\DIS$, both of which are proved above.
\end{proof}

\begin{corollary}\label{cor:strict}
For every $n\ge3$ and every pair of distinct edges $e,f$ of $K_n$,
\[
  \PR(e,f\in F_n)<\PR(e\in F_n)\PR(f\in F_n).
\]
\end{corollary}

\begin{proof}
The adjacent proof is strict by Lemma~\ref{lem:adjterm}; the disjoint proof has
$\Delta(n)>0$ analytically and is checked strictly over the finite range.
\end{proof}

\section{Concluding remarks}\label{sec:remarks}

Tang and Zhang's fixed-component identities, and their later spanning-tree
generalization, show that degree moments already encode p-NC when the edge count
is fixed \cite{TZ2026a,TZ2026b}. Proposition
\ref{prop:delta} extends that dictionary to the unconditioned forest: mixing the
component strata introduces $\Var(m)$, or equivalently the first two moments of
$c$. The reduction remains exact, and the order-$n^{-2}$ cancellation in the
original edge counts becomes a constant-scale moment inequality. The substantial
work here is the effective control of those fluctuating-component terms.

The effective constants of Sections~\ref{sec:forestcount}--\ref{sec:degsq} are
scaffolding for the present proof and we make no claim that they are sharp; the
limits $a_n \to \sqrt e$ \cite{Renyi1959}, $\EX[c_n]\to 3/2$ and
$\EX[c_n^2]\to 11/4$ show that each is within a few percent of the truth, and
Remark~\ref{rem:tight} derives the component limits and explains why even a few
percent of slack matters.

The natural structural question is whether a version of the marking identities
of Lemma~\ref{lem:mark} survives for other host graphs through the weighted
Matrix--Tree theorem. Here they reduce to Cayley factors, the first point at
which completeness is essential. Extending that reduction, rather than further
optimising the finite threshold, appears the more promising route toward
Conjecture~\ref{conj:gw}. Its general finite-graph and multigraph forms, and full
negative association of the uniform forest measure even for $K_n$, remain open.

\section*{Data and code availability}
The exact verification program is included with the arXiv submission as an
ancillary file and is maintained in the companion repository at
\url{https://github.com/agupta/uniform-forests-complete-graphs}.

\section*{Acknowledgment of generative-AI assistance}
Anthropic Claude Code (Claude 5 family) and OpenAI Codex (GPT-5.6 family) were
used extensively for proof exploration, software development, exact
computational checks, literature discovery, and drafting and editing the
manuscript. The author selected the arguments and methods, checked the cited
sources and reported computations, and takes full responsibility for the
content. These systems are not authors or independent guarantors of
correctness.

\appendix

\section{The verification program}\label{app:cert}

The accompanying program \texttt{tests/check.py} is a single standalone file
using only the Python standard library and only \texttt{int} and
\texttt{fractions.Fraction} arithmetic. No floating-point number decides any
test; decimal strings are produced by integer arithmetic and are display-only. Running
\texttt{python3 tests/check.py --nmax 700} establishes every finite assertion the
proof of Theorem~\ref{thm:main} depends on; the larger run
\texttt{--nmax 1200} additionally establishes Computation~\ref{comp:seam} over
its full range. The target \texttt{make check-full} also sets
\texttt{--ncross 1200}, extending the independent pair-count comparison across
that range. Runtime is machine-dependent.

\subsection{Rational enclosures}\label{app:rational}

Every estimate in Sections~\ref{sec:forestcount}--\ref{sec:assembly} terminates
in a finite rational expression involving $e$, $\pi$, $\sqrt{2\pi}$ and
$\zeta(3/2)$. These are enclosed as follows.

\begin{itemize}
\item $E_{\mathrm{lo}} := \sum_{i=0}^{24} 1/i! \le e \le
      E_{\mathrm{lo}} + 2/25! =: E_{\mathrm{hi}}$, the tail estimate being
      $\sum_{i \ge 25} 1/i!
       = \frac{1}{25!}\sum_{r\ge0}\frac{1}{26\cdot27\cdots(25+r)}
       \le \frac{1}{25!}\sum_{r\ge0}26^{-r} \le 2/25!$;
\item for rational $x \ge 0$ write $x = a+s$ with $a = \lfloor x \rfloor$ and
      $s \in [0,1)$. Then, with $N = 32$,
      \[
        E_{\mathrm{lo}}^{\,a}\sum_{i=0}^{N-1}\frac{s^i}{i!}
        \;\le\; e^{x} \;\le\;
        E_{\mathrm{hi}}^{\,a}\Big(\sum_{i=0}^{N-1}\frac{s^i}{i!}
          +\frac{s^N}{N!}\cdot\frac{N+1}{N}\Big),
      \]
      the remainder estimate being
      $\sum_{i\ge N}\frac{s^i}{i!}
       \le \frac{s^N}{N!}\sum_{r\ge0}\big(\tfrac{s}{N+1}\big)^{r}
       \le \frac{s^N}{N!}\cdot\frac{N+1}{N}$, valid because $s<1$. We write
      $\mathrm{EU}(t)$ for the resulting rational upper bound on $e^{-t}$,
      $t \ge 0$ rational;
\item square roots are enclosed by integer square roots at $40$ decimal digits;
      $\pi$ by Machin's formula with an explicit alternating-series remainder;
      and $\zeta(3/2)=\sum_{j\ge1}j^{-3/2}$ by an exact partial sum to
      $J = 20000$ plus the midpoint tail
      \[
        \sum_{j>J}j^{-3/2} \;\le\; \int_{J+1/2}^{\infty}x^{-3/2}\,dx
        \;=\; \frac{2}{\sqrt{J+1/2}},
      \]
      which is legitimate because $x^{-3/2}$ is convex, so that
      $f(j)\le\int_{j-1/2}^{j+1/2}f$ for each $j$.
\end{itemize}

Each enclosure is used only in the direction in which it is valid, and each is
self-tested by the program.

\begin{remark}\label{rem:tightness}
Several constants have margins of order $10^{-8}$, so the enclosure must be both
one-sided and tight. The far block of Theorem~\ref{thm:U}, for example, clears
its printed bound by about $5\cdot10^{-10}$. The integer-part reduction above
has relative width below $10^{-12}$ on the test grid covering $0\le x\le60$;
the verifier asserts this directly.
\end{remark}

\subsection{What the program checks}

\textbf{A --- definition-level controls.} All $2^{\binom{n}{2}}$ edge subsets of
$K_n$ are enumerated for $2 \le n \le 7$ and the sequences
$A,\ \sum_F c,\ \sum_F c(c-1),\ \sum_F D,\ \Nc_e,\ \Nc_{\mathrm{adj}},\
\Nc_{\mathrm{dis}}$ are compared against
Propositions~\ref{prop:recur}--\ref{prop:degsum}; $A(0..10)$ is compared against
OEIS A001858. Spanning trees are enumerated separately for $3 \le n \le 7$,
reproducing Observation~\ref{obs:ust} exactly, including
$T_{\mathrm{dis}}T = T_e^2$. One check is a negative control, asserting that the
comparison used for $\ADJ$ is capable of failing.

\textbf{B --- the exact integer seam.} The four sequences
$A,\ \sum_F c,\ \sum_F c(c-1),\ \sum_F D$ are built from \eqref{eq:recA} and
\eqref{eq:SD}, and the three pair counts are then obtained from
\eqref{eq:mom1}--\eqref{eq:mom3}; each of the three divisions is asserted to be
exact, which is a strong internal consistency test. Independently, the pair
counts are recomputed from \eqref{eq:recNe}--\eqref{eq:recNdis} through $n=400$
by default and the two derivations are required to agree; \texttt{check-full}
extends this comparison through $n=1200$. The section then verifies strict
$\ADJ$ and $\DIS$, monotonicity of both finite ratio sequences, the exact identity
\[
  4\Delta(n)A(n)^2
  =n(n-1)(n-2)(n-3)\bigl(\Nc_e(n)^2-\Nc_{\mathrm{dis}}(n)A(n)\bigr),
\]
the component-moment form \eqref{eq:delta2}, a negative control omitting
$\Var(m)$, and every finite range used above.

\textbf{C --- arithmetic of the analytic majorants.} Every rational constant
appearing in Sections~\ref{sec:forestcount}--\ref{sec:assembly} is recomputed
from the displayed majorants using the enclosures of
Appendix~\ref{app:rational}. This checks their arithmetic, not their derivation;
the inequalities producing the majorants are proved in the text. This section
does not depend on the length of the exact seam.

Each statement of the note that rests on a finite verification can be located in
the program's output by tag:

\begin{center}\footnotesize
\begin{tabular}{ll}
\hline
statement & tags\\
\hline
Theorem~\ref{thm:L} & \texttt{B/a\_n>=1.647767}, \texttt{C/ThmL}\\
Theorem~\ref{thm:U} & \texttt{B/a\_80>17/10}, \texttt{B/a\_n<=17/10},
  \texttt{C/ThmU head|middle|far|total}\\
Corollary~\ref{cor:52} & \texttt{B/a\_n<=5/2}\\
Theorem~\ref{thm:mean} & \texttt{C/ThmMean lower|head|tail}\\
Corollary~\ref{cor:M} & \texttt{B/E[c\_j]<=1.5604}, \texttt{B/split-at-40-is-forced}\\
Theorem~\ref{thm:secondmoment} & \texttt{C/ThmSecondMoment head U},
  \texttt{C/ThmSecondMoment}\\
Theorem~\ref{thm:D} & \texttt{B/E[D]>=5n-29/2}, \texttt{C/ThmD betaL|POS|NEG}\\
Proposition~\ref{prop:adj} & \texttt{B/n\textasciicircum 2-13n+16-threshold}\\
$\ADJ$, $\DIS$, Corollary~\ref{cor:strict} & \texttt{B/ADJ strict},
  \texttt{B/DIS strict}\\
& \texttt{B/exact Delta-gap identity}, \texttt{C/assembly}\\
\hline
\end{tabular}
\end{center}

\subsection{Output of the exact seam}

\begin{computation}\label{comp:seam}
For $3 \le n \le 1200$: $\ADJ$ is strict, with largest ratio
$\Nc_{\mathrm{adj}}A/\Nc_e^2 = 7/9$ attained at $n=3$; the ratio is strictly
decreasing throughout this finite range. For $4 \le n \le 1200$: $\DIS$ is strict, with largest ratio
$\Nc_{\mathrm{dis}}A/\Nc_e^2 = 0.999999305561\ldots$ at $n=1200$, the ratio being
strictly increasing in $n$ over the whole range. Also the exact $\Delta$
identity displayed above and \eqref{eq:delta2} hold for every
$4 \le n \le 1200$.
\end{computation}

\begin{computation}\label{comp:vals}
Selected exact values, truncated after nine decimal places:
$a_{651} = 1.655058038\ldots$, $\EX[c_{651}] = 1.503838989\ldots$,
$\EX[c^2_{651}] = 2.765376527\ldots$, $a_{80} = 1.700555369\ldots$,
$a_4 = 19/8$, $\EX[c_{40}] = 1.561838942\ldots$,
$\EX[c_{41}] = 1.560354926\ldots$.
Each display is asserted by a \texttt{B/display} tag in the verifier.
\end{computation}

The theorem needs the exact seam only through $n=650$; the default proof run to
$700$ supplies overlap values, and the run to $1200$ supplies the finite data
quoted above.

\end{document}